\documentclass[11pt]{amsart}

\usepackage{amssymb, amsmath, amsthm}

\newtheorem{thm}{Theorem}

 \newtheorem{lem}[thm]{Lemma}
 \newtheorem{prop}[thm]{Proposition}

    \theoremstyle{definition}
 \newtheorem{defn}[thm]{Definition}

 \theoremstyle{remark}
 
\begin{document}

\title[Symmetric bilinear form]{Symmetric bilinear form on a Lie algebra}

\author{Eun-Hee Cho and Sei-Qwon Oh}

\address{Department of Mathematics, Chungnam National  University, 99 Daehak-ro,   Yuseong-gu, Daejeon 34134, Korea}

\email{ehcho@cnu.ac.kr}
\email{sqoh@cnu.ac.kr}

\thanks{This research has been performed as a subproject of project Research for Applications of Mathematical Principles (No. C21501) and supported by the National Institute of Mathematics Science.}

\subjclass[2010]{17B67}

\keywords{Finite dimensional Lie algebra, Symmetric bilinear form}



\begin{abstract}
Let $\frak g$ be the finite dimensional simple Lie algebra associated to an indecomposable and  symmetrizable generalized Cartan matrix $C=(a_{ij})_{n\times n}$ of finite type and let  $\frak d$ be a finite dimensional Lie algebra related to  a quantum group  $D_{q,p^{-1}}(\frak g)$ obtained by Hodges, Levasseur and Toro \cite{HoLeT} by deforming the quantum group $U_q(\frak g)$.
 Here we see that $\frak d$ is a  generalization of  $\frak g$ and give a $\frak d$-invariant symmetric bilinear form on $\frak d$.
\end{abstract}

\maketitle

Let $C=(a_{ij})_{n\times n}$ be an indecomposable and  symmetrizable generalized Cartan matrix  of finite type and let $\frak g$ be the finite dimensional simple Lie algebra associated to $C$.
(Refer to  \cite[Chapter 2]{HoKa} and \cite[Chapter 1, 2, 4]{Kac}  for details.) Hodges, Levasseur and Toro  constructed a quantum group  $D_{q,p^{-1}}(\frak g)$ in \cite[Theorem 3.5]{HoLeT} by deforming the quantum group $U_q(\frak g)$, which is considered as a generalization of $U_q(\frak g)$, and obtained a Hopf dual $\Bbb C_{q,p}[G]$ of $D_{q,p^{-1}}(\frak g)$ in \cite[\S3]{HoLeT} that is a generalization of the Hopf algebra $\Bbb C_q[G]$ studied  in \cite{Jos} and \cite[\S3]{Tan}. The second author  constructed a finite dimensional Lie algebra $\frak d$ in \cite[Theorem~1.3]{Oh13} by using a skew symmetric bilinear form $u$ on a Cartan subalgebra $\frak h$ of $\frak g$. He showed in \cite[Proposition 1.4]{Oh13} that $\frak d$ is a generalized Lie bialgebra of the standard Lie bialgebra given in $\frak g$. Moreover he studied in  \cite[\S3 and \S5]{Oh13} the Poisson structure of the Hopf dual $\Bbb C[G]$ of the universal enveloping algebra $U(\frak d)$ that is considered as a Poisson version of $\Bbb C_{q,p}[G]$.
 In this note we see that $\frak d$ is a generalization of $\frak g$  (Proposition~\ref{GENERAL}) and find a $\frak d$-invariant symmetric bilinear form on $\frak d$ (Theorem~\ref{MCHOH}).
\medskip

We begin with explaining  the notations in \cite[1.1]{Oh13}. Let $C=(a_{ij})_{n\times n}$ be an indecomposable and  symmetrizable generalized Cartan matrix  of finite type.
  Hence there exists a diagonal matrix $D=\text{diag}(d_1,\ldots,d_n)$,  where all
  $d_i$ are positive integers,   such that the matrix $DC$ is symmetric positive definite.
  (Each $d_i$ is denoted by $s_i$ and $\epsilon_i^{-1}$ in  \cite[\S2.3]{HoKa} and \cite[Chapter 2]{Kac} respectively.)
Throughout the paper, we denote by
$$\begin{array}{ll}
\frak g=(\frak g, [\cdot,\cdot]_{\frak g}) &\text{the finite dimensional simple  Lie algebra over the} \\
&\text{complex number field  $\Bbb C$ associated to  $C$},\\
\frak h &\text{a Cartan subalgebra of $\frak g$ with simple roots $\alpha_1, \ldots, \alpha_n$}.
\end{array}$$
Choose $h_i\in\frak h$, $1\leq i\leq n$, such that
\begin{equation}\label{Z}
\alpha_j:\frak h\longrightarrow \Bbb C,\ \ \alpha_j(h_i)=a_{ij}\ \text{ for all }j=1,\ldots,n.
\end{equation}
Then $\{h_i\}_{i=1}^n$ forms  a basis of $\frak h$, since  $C$ has rank $n$, and
$\frak g$ is generated by $h_i$ and $x_{\pm\alpha_i}$, $i=1, \ldots , n$, with relations
$$\begin{array}{c}
[h_i,h_j]_{\frak g}=0,\  [h_i,x_{\pm\alpha_j}]_{\frak g}=\pm a_{ij}x_{\pm\alpha_j}, \ [x_{\alpha_i},x_{-\alpha_j}]_{\frak g}=\delta_{ij}h_i,\\
(\text{ad}_{x_{\pm\alpha_i}})^{1-a_{ij}}(x_{\pm\alpha_j})=0,\ \ i\neq j
\end{array}$$
by \cite[Definition 2.1.3]{HoKa}. (In \cite[Definition 2.1.3]{HoKa}, $x_{\alpha_i}$ and $x_{-\alpha_i}$ are denoted by $e_i$ and $f_i$ respectively.)
Denote by
$$\begin{array}{ll}
\frak n^+ &\text{the subspace of $\frak g$ spanned by root vectors with positive roots}\\
\frak n^- &\text{the subspace of $\frak g$ spanned by root vectors with negative roots}
\end{array}
$$
 and set $$\frak n=\frak n^-\oplus\frak n^+.$$
  Hence
$$\frak g=\frak h\oplus\frak n=\frak n^-\oplus\frak h\oplus\frak n^+.$$
Henceforth  we mean by $x_\alpha$ that $x_\alpha$ is a root vector of $\frak n$ with root $\alpha$.

By \cite[\S2.3]{HoKa}, there exists a nondegenerate  symmetric bilinear form $(\cdot|\cdot)$ on $\frak h^*$ given
by
\begin{equation}\label{Y}
(\alpha_i|\alpha_j)=d_ia_{ij}
\end{equation}
 for all $i,j=1,\ldots,n$.
This  form $(\cdot|\cdot)$ induces the isomorphism
$$\frak h^*\longrightarrow\frak h,\ \lambda\mapsto h_\lambda,$$ where $h_\lambda$ is  defined by
$$(\alpha_i|\lambda)=\alpha_i(h_\lambda)\ \ \text{ for all } i=1,\ldots,n.$$
(The isomorphism $\frak h^*\longrightarrow\frak h,  \lambda\mapsto h_\lambda$, is denoted by $\nu^{-1}$ in \cite[\S2.3]{HoKa}
and \cite[Chapter 2]{Kac}.)
Identifying $\frak h^*$ to $\frak h$   via $\lambda\mapsto h_\lambda$,
 $\frak h$ has a nondegenerate  symmetric bilinear form $(\cdot|\cdot)$ given by
 $$(\lambda| \mu)=(h_\lambda|h_\mu)=\lambda(h_\mu).$$
This is extended to
a nondegenerate $\frak g$-invariant symmetric bilinear form on $\frak g$ by \cite[Theorem 2.2 and its proof]{Kac} and \cite[(2.7) and Proposition 2.3.6]{HoKa}:
\begin{equation}\label{YA}
(h_i|h_j)=d_j^{-1}a_{ij},\ \ (h|x_{\alpha})=0,\ \ (x_{\alpha}|x_{\beta})=0\text{ if } \alpha+\beta\neq0,\ \
(x_{\alpha_i}|x_{-\alpha_j})=d_i^{-1}\delta_{ij}
\end{equation}
for  $h\in\frak h$, root vectors $x_\alpha,x_\beta$ and $i=1,\ldots,n$.

\begin{lem}\label{Oh13}\cite[Lemma 1.1]{Oh13}
For each positive root $\alpha$,
\begin{equation}\label{X}
[x_{\alpha},x_{-\alpha}]_{\frak g}=(x_\alpha|x_{-\alpha})h_{\alpha}.
\end{equation}
\end{lem}

 Let  $u=(u_{ij})$  be a skew symmetric $n\times n$-matrix with entries in $\Bbb C$. Then $u$ induces a skew symmetric bilinear (alternating)  form $u$
on $\frak h^*$ given by
$$u(\lambda,\mu):=\sum_{i,j} u_{ij} \lambda(h_i)\mu(h_j)$$
for any $\lambda,\mu\in \frak h^*$.
Hence there exists a unique linear map  $\Phi:\frak h^*\longrightarrow \frak h^*$ such that
$$u(\lambda,\mu)=(\Phi(\lambda)|\mu)$$
for any $\lambda,\mu\in \frak h^*$ since the form $(\cdot|\cdot)$ on $\frak h^*$ is nondegenerate. Set
$$ \Phi_+=\Phi + I,\ \ \  \Phi_-=\Phi -I,$$
where $I$ is the identity map on $\frak h^*$.  Thus
\begin{equation}\label{Phi}
\begin{array}{c}
(\Phi_+\lambda|\mu)=u(\lambda,\mu)+(\lambda|\mu)\\
(\Phi_-\lambda|\mu)=u(\lambda,\mu)-(\lambda|\mu)
\end{array}
\end{equation}
 for all $\lambda,\mu\in\frak h^*$.

Fix a vector space $\frak k$ isomorphic to $\frak h$ and let
\begin{equation}\label{phi}
\varphi:\frak h\longrightarrow\frak k
\end{equation}
be an isomorphism of vector spaces. For each $\lambda\in\frak h^*$, denote by $k_\lambda\in\frak k$ the  element $\varphi(h_\lambda)$. Let
$$\frak g':=\frak k\oplus\frak n=\frak n^-\oplus\frak k\oplus\frak n^+$$
 be the Lie algebra isomorphic to $\frak g$ such that each element $k_\lambda\in\frak k$ corresponds to the element  $h_\lambda$
 and each root vector $x_\alpha$ corresponds to  $x_\alpha$.  That is, $\frak g'$ is the Lie algebra with Lie bracket
\begin{equation}\label{Gprime}
\begin{array}{ll}
\ [k_\lambda,k_\mu]_{\frak g'}=0, &
[k_\lambda, x_\alpha]_{\frak g'}=(\alpha|\lambda)x_\alpha, \\
 \ [x_\alpha,x_\beta]_{\frak g'}=[x_\alpha,x_\beta]_{\frak g}\ \  (\alpha\neq-\beta),&
[x_\alpha,x_{-\alpha}]_{\frak g'}=\varphi([x_\alpha,x_{-\alpha}]_{\frak g}),
\end{array}
\end{equation}
where $\lambda,\mu\in\frak h^*$ and $x_\alpha,x_{-\alpha},x_\beta$ are root vectors with roots $\alpha,-\alpha,\beta$ respectively.

\begin{defn} \cite[Theorem 1.3]{Oh13} The vector space $\frak d:=\frak n^-\oplus\frak k\oplus \frak h\oplus \frak n^+$
is a Lie algebra with Lie bracket
\begin{eqnarray}
\ &[h_\lambda, h_\mu]=0, \ \ \ [h_\lambda,k_\mu]=0, \ \ \ [k_\lambda,k_\mu]=0,\label{Lie bracket1}\\
\ &[h_\lambda, x_\alpha]=-(\Phi_-\lambda|\alpha)x_\alpha,\label{Lie bracket2}\\
\ &[k_\lambda,x_\alpha]=(\Phi_+\lambda|\alpha)x_\alpha,\label{Lie bracket3}\\
\ &[x_\alpha,x_\beta]=2^{-1}([x_\alpha,x_\beta]_{\frak g}+[x_\alpha,x_\beta]_{\frak g'})\label{Lie bracket4}
\end{eqnarray}
for all $h_\lambda, h_\mu\in\frak h$,  $k_\lambda, k_\mu\in\frak k$ and  root vectors $x_\alpha,x_\beta\in\frak n=\frak n^-\oplus\frak n^+$.
\end{defn}

The Lie algebra $\frak d$ is a generalization of $\frak g$ as seen in the following proposition.

\begin{prop}\label{GENERAL}
(a)  Let $\frak l$ be the subspace of $\frak d$ spanned by all $h_\lambda-k_\lambda, \lambda\in\text{rad}(u)$, where
$$\text{rad}(u)=\{\lambda\in\frak h^*|u(\lambda,\frak h^*)=0\}.$$ Then $\frak l$ is a solvable ideal of $\frak d$. In particular, if   $u=0$ then  $\frak l$  is the maximal solvable ideal of $\frak d$ and  $\frak d/\frak l$
is isomorphic to $\frak g$.

(b) Let $\frak m$ be the subspace of $\frak d$ spanned by all root vectors $x_\alpha$ and all $h_\lambda+k_\lambda, \lambda\in\frak h^*$. Then
$\frak m$ is an ideal of $\frak d$ isomorphic to $\frak g$.

\end{prop}

\begin{proof}
(a) Note that
$$[h_\lambda-k_\lambda, x_\alpha]=-2u(\lambda,\alpha)x_\alpha$$ for any root vector $x_\alpha$. Hence
if  $\lambda\in\text{rad}(u)$ then $[h_\lambda-k_\lambda, x_\alpha]=0$ for all root vectors $x_\alpha$. It follows that $\frak l$ is an ideal of $\frak d$. Clearly $\frak l$ is solvable by (\ref{Lie bracket1}).

Suppose that $u=0$. Then $\text{rad}(u)=\frak h^*$ and thus the ideal $\frak l$ is the subspace spanned by all $h_\lambda-k_\lambda, \lambda\in\frak h^*$, which is solvable.  It is easy to see that $\frak g\cong\frak d/\frak l$ by (\ref{Lie bracket4}) and (\ref{Gprime}). Hence $\frak d/\frak l$ is simple and thus $\frak l$ is the unique maximal solvable ideal of $\frak d$.

(b)  Note that
$$[h_\lambda+k_\lambda, x_\alpha]=2(\lambda|\alpha)x_\alpha$$ for any root vector $x_\alpha$ and
$$[x_\alpha,x_{-\alpha}]=2^{-1}(x_\alpha|x_{-\alpha})(h_\alpha+k_\alpha)$$ for any positive root $\alpha$ by Lemma~\ref{Oh13} and (\ref{Lie bracket4}). Hence $\frak m$ is an ideal of $\frak d$ by (\ref{Lie bracket1})-(\ref{Lie bracket4}). Moreover $\frak m$ is isomorphic to $\frak g$ since
the linear map from $\frak m$ into $\frak g$ defined by
$$h_\lambda+k_\lambda\mapsto 2h_\lambda,\ \ \  x_\alpha\mapsto x_\alpha\ \ \ (\lambda\in\frak h^*)$$
 is a Lie algebra  isomorphism.
\end{proof}


We give a $\frak d$-invariant symmetric bilinear form on $\frak d$ as in the following theorem.

\begin{thm}\label{MCHOH}
Define a bilinear form $(\cdot|\cdot)_{\frak d}$ on $\frak d$ by
\begin{eqnarray}
&(h| h')_{\frak d}=2(h|h'),& (h|x)_{\frak d}=(h|x)=0,\label{Cho1}\\
&(x|h)_{\frak d}=(x|h)=0,&(x|x')_{\frak d}=(x|x'),\label{Cho2}\\
 &(h_\lambda|k_\mu)_{\frak d}=-2u(\lambda,\mu),&(k_\mu|h_\lambda)_{\frak d}=2u(\mu,\lambda),\label{Cho3}\\
 &(k|x)_{\frak d}=(\varphi^{-1}(k)|x)=0,&(x|k)_{\frak d}=(x|\varphi^{-1}(k))=0,\label{Cho4} \\
 &(k|k')_{\frak d}=2(\varphi^{-1}(k)|\varphi^{-1}(k'))&\label{Cho5}
\end{eqnarray}
for $h,h', h_\lambda\in\frak h$, $x,x'\in\frak n^+\oplus\frak n^-$, $k,k', k_\mu\in\frak k$, where $\varphi$ is the isomorphism given in (\ref{phi}). Then $(\cdot|\cdot)_{\frak d}$ is a $\frak d$-invariant symmetric
 bilinear form.
\end{thm}

\begin{proof}
Since  $(\cdot|\cdot)$ is a symmetric bilinear form on $\frak g$ and $u$ is a skew symmetric bilinear form on $\frak h^*$, $(\cdot|\cdot)_{\frak d}$ is clearly symmetric by (\ref{Cho1})-(\ref{Cho5}).

Let us show that $(\cdot|\cdot)_{\frak d}$ is $\frak d$-invariant, that is,
$$(a|[b,c])_{\frak d}=([a,b]|c)_\frak d$$
for all $a,b,c\in \frak d$.

Case I. $a=x_\alpha,b=x_\beta, c=x_\gamma$:

(a) $\alpha+\beta+\gamma\neq0$:
Note that
 $(x_{\alpha'}|x_{\beta'})_{\frak d}=(x_{\alpha'}|x_{\beta'})=0$  for root vectors $x_{\alpha'},x_{\beta'}$ of $\frak n$ with $\alpha'+\beta'\neq0$ by (\ref{YA}) and (\ref{Cho2}). Hence
$$([x_\alpha,x_\beta]| x_\gamma)_{\frak d}=0$$
 since $[x_\alpha,x_\beta]\in\frak n$ if $\alpha+\beta\neq0$ and $[x_\alpha,x_\beta]\in \frak h\oplus \frak k$ if $\alpha+\beta=0$ by (\ref{Cho1}) and  (\ref{Cho4}). Similarly
$$(x_\alpha|[x_\beta, x_\gamma])_{\frak d}=0$$
since $[x_\beta,x_\gamma]\in\frak n$ if $\beta+\gamma\neq0$ and $[x_\beta,x_\gamma]\in\frak h\oplus\frak k$ if $\beta+\gamma=0$  by (\ref{Cho2}) and  (\ref{Cho4}).
Hence
$$([x_\alpha,x_\beta]| x_\gamma)_{\frak d}=0=(x_\alpha|[x_\beta, x_\gamma])_{\frak d}.$$

 (b) $\alpha+\beta+\gamma=0$: Since $\alpha+\beta=-\gamma$, we have $[x_\alpha,x_\beta]=[x_\alpha,x_\beta]_{\frak g}$ by (\ref{Gprime}) and (\ref{Lie bracket4}). Similarly,
 since $\beta+\gamma=-\alpha$, we also have $[x_\beta,x_\gamma]=[x_\beta,x_\gamma]_{\frak g}$ by (\ref{Gprime}) and (\ref{Lie bracket4}). Hence  we have that
 $$([x_\alpha,x_\beta]|x_\gamma)_{\frak d}=([x_\alpha,x_\beta]_{\frak g}|x_\gamma)=(x_\alpha|[x_\beta,x_\gamma]_{\frak g})
 =(x_\alpha|[x_\beta,x_\gamma])_{\frak d}$$
 by (\ref{Cho2}).

Case II. $a=h_\lambda\in \frak h$, $b=x_\beta$,  $c=x_\gamma$:

 (a) $\beta+\gamma=0$: Note that $x_\gamma=x_{-\beta}$.
Since $[h_\lambda,x_\beta]=-(\Phi_-\lambda|\beta)x_\beta$ by (\ref{Lie bracket2}) and
$$[x_\beta,x_{-\beta}]=2^{-1}([x_\beta,x_{-\beta}]_{\frak g}+[x_\beta,x_{-\beta}]_{\frak g'})=2^{-1}(x_\beta|x_{-\beta})h_{\beta}+2^{-1}(x_\beta|x_{-\beta})k_{\beta}$$
 by (\ref{Lie bracket4}) and (\ref{X}),
we have that
$$
\begin{aligned}
([h_\lambda,x_\beta]| x_{-\beta})_{\frak d}&=-(\Phi_-\lambda|\beta)(x_\beta| x_{-\beta})_{\frak d}&&(\text{by } (\ref{Lie bracket2}))\\
&=(\lambda|\beta)(x_\beta| x_{-\beta})-u(\lambda,\beta)(x_\beta| x_{-\beta})&&(\text{by }(\ref{Phi}), (\ref{Cho2}))\\
&=2^{-1}(h_\lambda| (x_\beta| x_{-\beta})h_\beta)_{\frak d}+2^{-1}(h_\lambda|(x_\beta| x_{-\beta})k_\beta)_{\frak d}\ &&(\text{by }(\ref{Cho1}), (\ref{Cho3}))\\
&=2^{-1}(h_\lambda|[x_\beta, x_{-\beta}]_{\frak g}+[x_\beta, x_{-\beta}]_{\frak g'})_{\frak d}\ &&(\text{by }(\ref{X}))\\
&=(h_\lambda|[x_\beta, x_{-\beta}])_{\frak d}.\ &&(\text{by }(\ref{Lie bracket4}))
\end{aligned}$$

 (b) $\beta+\gamma\neq0$:
Since $[x_\beta,x_\gamma]\in \frak n$, we have that
$$
\begin{aligned}
([h_\lambda,x_\beta]| x_{\gamma})_{\frak d}&=-(\Phi_-\lambda|\beta)(x_\beta| x_{\gamma})_{\frak d}=0=(h_\lambda|[x_\beta,x_\gamma])_{\frak d}
\end{aligned}$$
by (\ref{YA}), (\ref{Lie bracket2}), (\ref{Cho1}) and (\ref{Cho2}).

Case III. $a=k_\lambda\in \frak k$, $b=x_\beta$,  $c=x_\gamma$:
This case is to be shown as in Case II. We repeat it for completion.

 (a) $\beta+\gamma=0$: Note that $x_\gamma=x_{-\beta}$.
Since
$$[x_\beta,x_{-\beta}]=2^{-1}([x_\beta,x_{-\beta}]_{\frak g}+[x_\beta,x_{-\beta}]_{\frak g'})=2^{-1}(x_\beta|x_{-\beta})h_{\beta}+2^{-1}(x_\beta|x_{-\beta})k_{\beta}$$
 by (\ref{Lie bracket4}) and (\ref{X}),
we have that
$$
\begin{aligned}
([k_\lambda,x_\beta]| x_{-\beta})_{\frak d}&=(\Phi_+\lambda|\beta)(x_\beta| x_{-\beta})_{\frak d}&&(\text{by }(\ref{Lie bracket3}))\\
&=(\lambda|\beta)(x_\beta| x_{-\beta})+u(\lambda,\beta)(x_\beta| x_{-\beta})&&(\text{by }(\ref{Phi}),(\ref{Cho2}))\\
&=2^{-1}(k_\lambda| (x_\beta| x_{-\beta})k_\beta)_{\frak d}+2^{-1}(k_\lambda|(x_\beta| x_{-\beta})h_\beta)_{\frak d}&&(\text{by }(\ref{Cho5}), (\ref{Cho3}))\\
&=2^{-1}(k_\lambda|[x_\beta, x_{-\beta}]_{\frak g}+[x_\beta, x_{-\beta}]_{\frak g'})_{\frak d}&&(\text{by }(\ref{X}))\\
&=(k_\lambda|[x_\beta, x_{-\beta}])_{\frak d}.&&(\text{by }(\ref{Lie bracket4}))
\end{aligned}$$

(b) $\beta+\gamma\neq0$:
Since $[x_\beta,x_\gamma]\in\frak n$, we have that
$$
\begin{aligned}
([k_\lambda,x_\beta]| x_{\gamma})_{\frak d}&=(\Phi_+\lambda|\beta)(x_\beta| x_{\gamma})_{\frak d}=0=(k_\lambda|[x_\beta,x_\gamma])_{\frak d}
\end{aligned}$$
by (\ref{YA}), (\ref{Lie bracket3}), (\ref{Cho2}) and (\ref{Cho4}).

Case IV.  $a=x_\beta$, $b=x_\gamma$, $c=h_\lambda\in \frak h$: Since  $(\cdot|\cdot)_{\frak d}$ is symmetric,  we have that
$$
([x_\beta,x_\gamma]|h_\lambda)_{\frak d}=-(h_\lambda|[x_\gamma,x_\beta])_{\frak d}
=-([h_\lambda,x_\gamma]|x_\beta)_{\frak d}=(x_\beta|[x_\gamma,h_\lambda])_{\frak d}
$$    by Case II.

Case V.  $a=x_\beta$, $b=x_\gamma$, $c=k_\lambda\in \frak k$: Since  $(\cdot|\cdot)_{\frak d}$ is  symmetric,  we have that
$$
([x_\beta,x_\gamma]|k_\lambda)_{\frak d}=-(k_\lambda|[x_\gamma,x_\beta])_{\frak d}
=-([k_\lambda,x_\gamma]|x_\beta)_{\frak d}=(x_\beta|[x_\gamma,k_\lambda])_{\frak d}
$$    by Case III.

Case VI.  $a=x_\beta$, $b=h_\lambda\in \frak h$,  $c=x_\gamma$: Since $(\cdot|\cdot)_{\frak d}$ is symmetric, we have that
$$\begin{aligned}
([x_\beta,h_\lambda]|x_\gamma)_{\frak d}&=-([h_\lambda,x_\beta]|x_\gamma)_{\frak d}=-(h_\lambda|[x_\beta,x_\gamma])_{\frak d}\\
&=(h_\lambda|[x_\gamma,x_\beta])_{\frak d}=([h_\lambda,x_\gamma]|x_\beta)_{\frak d}=(x_\beta|[h_\lambda,x_\gamma])_{\frak d}
\end{aligned}$$
by Case II.

Case VII.  $a=x_\beta$, $b=k_\lambda\in \frak k$,  $c=x_\gamma$: Since $(\cdot|\cdot)_{\frak d}$ is symmetric, we have that
$$\begin{aligned}
([x_\beta,k_\lambda]|x_\gamma)_{\frak d}&=-([k_\lambda,x_\beta]|x_\gamma)_{\frak d}=-(k_\lambda|[x_\beta,x_\gamma])_{\frak d}\\
&=(k_\lambda|[x_\gamma,x_\beta])_{\frak d}=([k_\lambda,x_\gamma]|x_\beta)_{\frak d}=(x_\beta|[k_\lambda,x_\gamma])_{\frak d}
\end{aligned}$$
by Case III.

Case VIII.  $a\in\frak h\oplus \frak k$, $b\in\frak h\oplus \frak k$,  $c=x_\gamma$:
Clearly we   have  that
$$([a,b]|x_\gamma)_{\frak d}=0=(a|[b,x_\gamma])_{\frak d}$$
 by (\ref{Lie bracket1}), (\ref{Cho1}) and (\ref{Cho4}).

Case IX:  $a\in\frak h\oplus \frak k$,  $b=x_\gamma$, $c\in\frak h\oplus \frak k$:
Clearly we   have  that
$$([a,x_\gamma]|c)_{\frak d}=0=(a|[x_\gamma,c])_{\frak d}$$
by (\ref{Cho1}), (\ref{Cho2}) and   (\ref{Cho4}).

Case X. $a=x_\gamma$,    $b\in\frak h\oplus \frak k$,   $c\in\frak h\oplus \frak k$:
Clearly we   have  that
$$([x_\gamma,b]|c)_{\frak d}=0=(x_\gamma|[b,c])_{\frak d}$$
by (\ref{Cho2}), (\ref{Cho4}) and (\ref{Lie bracket1}).

Case XI. $a\in\frak h\oplus\frak k$,    $b\in\frak h\oplus \frak k$,   $c\in\frak h\oplus \frak k$:
Clearly we   have  that
$$([a,b]|c)_{\frak d}=0=(a|[b,c])_{\frak d}$$
by (\ref{Lie bracket1}).
\end{proof}


\bibliographystyle{amsplain}

\providecommand{\bysame}{\leavevmode\hbox to3em{\hrulefill}\thinspace}
\providecommand{\MR}{\relax\ifhmode\unskip\space\fi MR }
\providecommand{\MRhref}[2]{%
  \href{http://www.ams.org/mathscinet-getitem?mr=#1}{#2}
}
\providecommand{\href}[2]{#2}


\end{document}